\documentclass[11pt]{amsart}
\usepackage[left=2.2cm,tmargin=2.5cm,bmargin=2.5cm,right=2.2cm]{geometry}
\usepackage{amsfonts}
\usepackage{amssymb}
\usepackage{hyperref}
\usepackage{stmaryrd}
\usepackage{mathrsfs}
\usepackage[all]{xy}
\usepackage[capitalise]{cleveref}
\usepackage{amscd}
\usepackage{physics}
\usepackage{enumitem}
\usepackage{graphicx}
\usepackage[utf8]{inputenc}
\usepackage{pdfpages}
\usepackage{listings}

\newtheorem{thm}{Theorem}[section]
\newtheorem{thm*}{Theorem}
\newtheorem{lem}[thm]{Lemma}
\newtheorem{lem*}[thm*]{Lemma}
\newtheorem{cor}[thm]{Corollary}
\newtheorem{cor*}[thm*]{Corollary}
\newtheorem{prop}[thm]{Proposition}
\newtheorem{prop*}[thm*]{Proposition}

\newtheorem{conj*}[thm*]{Conjecture}

\newtheorem{defn*}[thm*]{Definition}

\newtheorem{ex*}[thm*]{Example}
\newtheorem{rmk}[thm]{Remark}
\newtheorem{rmk*}[thm*]{Remark}

\newtheorem{clm*}[thm*]{Claim}
\newtheorem{question}[thm]{Question}

\numberwithin{equation}{section}

\begin{document}

\title{subgroups with all finite preimages isomorphic are conjugate}

\author[Karshon]{Ido Karshon}
\author[Lubotzky]{Alexander Lubotzky}
\author[McReynolds]{D. B. McReynolds}
\author[Reid]{Alan W. Reid}
\author[Shusterman]{Mark Shusterman} 

\address{Faculty of Mathematics and Computer Science, Weizmann
Institute of Science, 234 Herzl Street, Rehovot 76100, Israel.}
\email{ido.karshon@gmail.com}

\address{Faculty of Mathematics and Computer Science, Weizmann
Institute of Science, 234 Herzl Street, Rehovot 76100, Israel.}
\email{alex.lubotzky@mail.huji.ac.il}

\address{Department of Mathematics, Purdue University, West Lafayette, IN 47907}
\email{dmcreyno@purdue.edu}

\address{Department of Mathematics, Rice University, Houston, TX 77005. School of Mathematics, Korea Institute for Advanced Study (KIAS), Seoul, 02455, Korea.}
\email{alan.reid@rice.edu}

\address{Faculty of Mathematics and Computer Science, Weizmann
Institute of Science, 234 Herzl Street, Rehovot 76100, Israel.}
\email{mark.shusterman@weizmann.ac.il}

\begin{abstract}
We show that for non-conjugate subgroups $G_1$ and $G_2$ of a finite group $G$ there exists an extension of $G$ (by a finite group) in which the pre-images of $G_1$ and $G_2$ are not isomorphic. This allows us to show that $\mathbb Z$--coset equivalent subgroups of a finite group are not necessarily isomorphic, answering a question of Dipendra Prasad. We also indicate connections to profinite rigidity, anabelian geometry, mapping class groups, and non-arithmetic lattices in Lie groups.
\end{abstract}



\maketitle

\section{Introduction}

\subsection{Main Result}

Let $G$ be a group, and let $G_1$, $G_2$ be subgroups of $G$. If $G_1$ and $G_2$ are conjugate, then for every group $\widetilde G$ and a surjective group homomorphism $\varphi \colon \widetilde G \to G$ the pre-images $\widetilde G_1 = \varphi^{-1}(G_1)$ and $\widetilde G_2 = \varphi^{-1}(G_2)$ are conjugate in $\widetilde G$ (by any element of $\widetilde G$ that is mapped by $\varphi$ to an element that conjugates $G_1$ to $G_2$). In particular $\widetilde G_1$ and $\widetilde G_2$ are isomorphic. Our main result says that if $G_1$ and $G_2$ are non-conjugate, and $G$ is finite, we can make $\widetilde G_1$ and $\widetilde G_2$ not isomorphic and finite.

\begin{thm} \label{MainRes}
Let $G$ be a finite group. Then there exists a finite group $\widetilde G$ and a surjective group homomorphism $\varphi \colon \widetilde G \to G$ such that for any two non-conjugate subgroups $G_1, G_2 \le G$, the pre-images $\widetilde G_1 = \varphi^{-1}(G_1)$ and $\widetilde G_2 = \varphi^{-1}(G_2) $ are not isomorphic.
\end{thm}

Moreover, we can arrange for the kernel of $\varphi$ to be supersolvable, and in particular, if $G$ is solvable, then so is $\widetilde G$. Our proof of \cref{MainRes} builds on the following result posted originally on MathOverflow, see \cite{cornulier-original}.

\begin{thm}[Y. Cornulier]\label{thm:cornulier} 
If $G$ is a finite group, then there exists a finite group $N$ such that $G \cong \operatorname{Out}(N)$.
\end{thm}

In fact, the explicit construction of $N$ (or rather some of its properties - see \cref{thm:cornulier-with-properties}) plays a role in our argument.  Our understanding of it was greatly aided by Sambale's exposition of Cornulier's construction in \cite{group-theoretic-cornulier}.

We also obtain a variant of \cref{MainRes} in the category of all (not only finite) groups.

\begin{thm} \label{thm:main:infinite}
Let $G$ be a group. Then there exists a group $\widetilde G$ and a surjective group homomorphism $\varphi \colon \widetilde G \to G$ such that for any two non-conjugate subgroups $G_1,G_2 \le G$, the pre-images $\varphi^{-1}(G_1)$ and $\varphi^{-1}(G_2) $ are not isomorphic.
\end{thm}

The proof of \cref{thm:main:infinite} uses the following realization theorem from \cite{DGG}.

\begin{thm}[Droste--Giraudet--G\"obel]\label{thm:DGG}
For every group $G$ there exists a nonabelian simple group $S$ such that
$\mathrm{Out}(S) \cong G$.
\end{thm}

\subsection{Application to $R$--coset equivalence}

Let $R$ be a commutative ring with identity and let $G$ be a group with subgroups $G_1,G_2$. We say that $G_1, G_2$ are {\em $R$--coset equivalent} if $R[G/G_1] \cong R[G/G_2]$ as $RG$--modules. A trivial example of $R$--coset equivalent subgroups is provided by conjugate subgroups. Also, the property of $R$--coset equivalence is preserved under pullbacks by epimorphisms, see Lemma \ref{L:RCos}.

When $R=\mathbb{Q}$, such pairs (also known as {\em Gassmann/Sunada pairs}) have been much studied since they arise naturally in the study of arithmetic equivalence and Dedekind zeta functions of number fields (see \cite{Perlis} and the book \cite{Klingen}) and were also used by Sunada \cite{Sunada} in his construction of isospectral Riemannian manifolds. There are, by now, many constructions of Gassmann pairs and, indeed,  there are many examples where $G_1$ and $G_2$ are not isomorphic; e.g. from \cite[Lemma 3.1]{Sp} if $G$ is a group containing the alternating group $A_{27}$ and $G_1 =(\mathbb{Z}/3\mathbb{Z})^3$ and $G_2 =$ the $3\times 3$ Heisenberg group over $\mathbb{Z}/3\mathbb{Z}$, then these can be embedded in $A_{27}$ such that $G_1$ and $G_2$ are $\mathbb{Q}$--coset equivalent.

More recently, applications of the case when $R=\mathbb{Z}$ have been explored: for example, the existence of a $\mathbb{Z}$--coset equivalent pair of non-conjugate subgroups found by L. Scott in \cite{Scott} (see below for a discussion) was used by D. Prasad in \cite{Prasad} to produce non-isomorphic number fields with isomorphic multiplicative groups and a bijection between their abelian extensions. The paper \cite{Prasad} also produces non-isomorphic Riemann surfaces with isomorphic Jacobians when viewed without their polarizations.  Other applications can be found in the paper \cite{AKMS} which used non-conjugate $\mathbb{Z}$--coset equivalent pairs to provide examples of non-isomorphic algebraic surfaces with isomorphic Jacobians as well as proving general results relating the cohomology of spaces constructed from such pairs using covering space theory. The paper \cite{GM} goes further and showed, among other things, the existence of non-isomorphic number fields for which the maximal pronilpotent quotients of the associated absolute Galois groups are isomorphic. These spaces also have the same nilpotent representation theory (which we will not describe carefully here).

However, unlike the case of  $\mathbb{Q}$--coset equivalent subgroups, it is worth emphasizing that there are very few examples of non-conjugate $\mathbb{Z}$--coset equivalent pairs known. The first non-conjugate pair was that found by Scott in \cite{Scott} (mentioned above): in this case $G = \mathrm{PSL}(2,29)$ and $G_1,G_2$ are a pair of non-conjugate  subgroups (of index $203$) isomorphic to $A_5$. The only other known method for producing additional examples from this one pair is via pullbacks by epimorphisms to $\mathrm{PSL}(2, 29)$, or products of such. Scott in \cite{Scott} also observed that when $p \equiv \pm 29 \mod 120$, there exists a pair of non-conjugate $A_5$ subgroups of $\mathrm{PSL}(2,p)$ that Scott showed satisfies a necessary condition to be $\mathbb{Z}$--coset equivalent but this has only been verified to be $\mathbb{Z}$--coset equivalent when $p=29$.  

Some motivation for this note comes from a question raised in Prasad's paper \cite{Prasad} and explored by Sutherland in \cite{Sutherland}:

\begin{question}[D. Prasad]\label{finite_iso} 
If $G$ is a finite group and $G_1,G_2 \leq G$ are $\mathbb{Z}$--coset equivalent, are $G_1,G_2$ necessarily isomorphic?
\end{question}

The following is an immediate consequence of \cref{MainRes}.

\begin{thm} \label{MainCor}
Question \ref{finite_iso} has a negative answer.
\end{thm}

Indeed, take for instance $G = \mathrm{PSL}(2,29)$ and the non-conjugate subgroups $G_1, G_2$ from \cite{Scott} that are $\mathbb Z$--coset equivalent, \cref{MainRes} furnishes us with a finite group $\widetilde G$ and a surjective group homomorphism $\varphi \colon \widetilde G \to G$ for which $\varphi^{-1}(G_1)$ and $\varphi^{-1}(G_2)$ are, in view of \cref{L:RCos}, $\mathbb Z$--coset equivalent subgroups that are not isomorphic.

In light of this, we propose the following version of Question \ref{finite_iso}.

\begin{question}
If $G$ is a finite group and $G_1,G_2 \leq G$ are $\mathbb{Z}$--coset equivalent subgroups that contain no non-trivial normal subgroup of $G$, are $G_1,G_2$ necessarily isomorphic?
\end{question}

Given a finite group $G$ and non-conjugate subgroups $G_1$ and $G_2$ we are also interested in the question of which (finite) groups $\widetilde G$ admit a surjective group homomorphism onto $G$ such that the pre-images of $G_1$ and $G_2$ are not isomorphic. \cref{MainRes} and its proof provide us with some information in this direction, and in the rest of the introduction we explore other strategies of producing a suitable $\widetilde G$ and its ensuing properties.

\subsection{Anabelian Geometry}

Suppose that we can realize $G$ as the Galois group of a (finite) Galois extension $K$ of $\mathbb Q$, and let $\widetilde G$ be the absolute Galois group of $\mathbb Q$ so that we have a natural (continuous) surjective homomorphism from the profinite group $\widetilde G$ onto $G$ given by restriction of automorphisms to $K$. Then the pre-images $\widetilde{G}_1$ and $\widetilde G_2$ of $G_1$ and $G_2$ are not isomorphic. Indeed, if they were isomorphic, it would follow from the Neukirch--Uchida theorem (see for instance \cite{KS}) that they are conjugate in $\widetilde G$, so we would get that $G_1$ and $G_2$ are conjugate, contrary to our assumption.

While it is possible to use in this argument variants of the Neukirch--Uchida theorem (see for instance \cite[Introduction]{KS}) which work with (much smaller) quotients of $\widetilde G$ in place of $\widetilde G$ itself, it is unclear how to deduce from this a result such as Theorem \ref{MainRes} where the groups are finite (and not merely profinite, contrast with Proposition \ref{CorProfComp}). It would however be interesting to see if, even just in the realm of profinite groups, one can obtain additional variants of \cref{MainRes} using other results from anabelian geometry (that perhaps do not necessitate the, at present conjectural, realization of $G$ as a Galois group). These could be, for instance, variants of the Neukirch--Uchida theorem over function fields.

We mention also that from the perspective of the Neukirch--Uchida theorem, the supersolvability of the kernel of $\varphi$ from (the proof of) \cref{MainRes} seems related to the question, raised in \cite{KS}, of whether a number field can be recovered from the maximal prosupersolvable quotient of its absolute Galois group. The relevance of coset equivalence to that perspective is considered in \cite{GM}.

\subsection{Profinite Completions}

In some cases it is more natural/convenient to construct first a profinite group $\widetilde G$ having the desired properties. For the purpose of passing from profinite groups to finite groups in our constructions of $\widetilde G$ we use the following consequence of \cref{SmallGroupProp}.

\begin{prop} \label{CorProfComp}
Let $\Gamma$ be a finitely generated group, let $G$ be a finite group with subgroups $G_1, G_2$, and let $\varphi \colon \Gamma \to G$ be a surjective homomorphism. Suppose that the profinite completions of $\varphi^{-1}(G_1)$ and $\varphi^{-1}(G_2)$ are not isomorphic. Then $\varphi$ factors through a surjective homomorphism $\Gamma \to H$ to a finite group in which the inverse images of $G_1$ and $G_2$ are not isomorphic.
\end{prop}

The extra information we may obtain from such a statement, and perhaps not from \cref{MainRes}, is that the finite group $\widetilde G$ in which we need to take the preimages of $G_1$ and $G_2$ for them not to be isomorphic is a quotient of $\Gamma$. With this in mind, in \S  \ref{Proof_C1}, we prove the following alternative to \cref{MainRes} in a special case of a certain group $\Gamma$ that admits a group homomorphism onto $\mathrm{PSL}(2, 29)$.

\begin{thm}\label{C:1}
Let $G=\mathrm{PSL}(2,29)$, and let $G_1$, $G_2$ be non-conjugate subgroups of $G$ isomorphic to $A_5$. Let $T \subset \mathbb{H}^3$ be the tetrahedron described by the $(3,5,3)$-Coxeter diagram, and let $\Gamma$ denote the subgroup of index $2$ in the group generated by reflections in the faces of $T$ consisting of orientation-preserving isometries. Then there exists a finite quotient group $Q$ of $\Gamma$ and an epimorphism $Q \to G$ for which the pre-images of $G_1$ and $G_2$ are not isomorphic.
\end{thm}

Our discussion is closely related to the subject of profinite rigidity - see for instance \cite{Rei}. Indeed, if in \cref{C:1} the finite-index subgroups of $\Gamma$ are determined up to isomorphism by their profinite completion, we can weaken the assumption to saying that the (discrete) groups $\varphi^{-1}(G_1)$ and $\varphi^{-1}(G_2)$ are not isomorphic.

If we also know that finite-index subgroups of $\Gamma$ are isomorphic only if they are conjugate, then we obtain a version of \cref{MainRes} for a suitable finite quotient of $\Gamma$. There are at least two settings where this can potentially be applied.

One possible setting is that of mapping class groups of closed orientable surfaces: while their finite-index subgroups are isomorphic only if they are conjugate in view of \cite[Theorem 5]{Iv}, we do not yet have the required profinite rigidity results.

A similar possible setting is that of maximal non-arithmetic lattices in Lie groups such as $O(n,1)/\{\pm 1\}$ that exist by \cite{GPS} and \cite{Margulis}. In this case, by work of Margulis, there is a unique maximal lattice in the commensurability class and for any pair of finite index subgroups, they are isomorphic as abstract groups if and only if they are conjugate in the maximal lattice. Profinite rigidity remains a possibility here as well.

\section{Permanence of Coset Equivalence Under Pullbacks}

\begin{lem} \label{L:RCos}
If $R$ is a ring, $G_1,G_2 \leq G$ are $R$--coset equivalent, and $\psi\colon \Gamma \to G$ is a surjective homomorphism of groups with $\Gamma_j = \psi^{-1}(G_j)$ for $j=1,2$, then $\Gamma_1,\Gamma_2$ are $R$--coset equivalent as well.
\end{lem}

\begin{proof}
Note that $\Gamma/\Gamma_j \cong G/G_j$ as $\Gamma$--sets for $j=1,2$ where the latter acts via $\psi$. In particular, for any ring $R$, we have $R[\Gamma/\Gamma_j] \cong R[G/G_j]$ as $R[\Gamma]$--modules. By assumption $R[G/G_1] \cong R[G/G_2]$ as $R[G]$--modules and hence as $R[\Gamma]$--modules. Combining these two observations, we see that 
\[ R[\Gamma/\Gamma_1] \cong R[G/G_1] \cong R[G/G_2] \cong R[\Gamma/\Gamma_2] \]
as $R[\Gamma]$--modules as claimed.
\end{proof}

\section{From Profinite to Finite Groups}

\begin{prop} \label{SmallGroupProp}
Let $G$ be a finite group, let $G_1$, $G_2$ be subgroups of $G$, and let $\varphi \colon \widetilde G \to G$ be a continuous surjective homomorphism of profinite groups for which the profinite groups $\varphi^{-1}(G_1)$ and $\varphi^{-1}(G_2)$ are not isomorphic. Assume that there are only finitely many open subgroups of any given index in $\widetilde G$. Then there exists a continuous surjective homomorphism $\psi \colon \widetilde G \to H$ onto a finite group, and a group homomorphism $\theta \colon H \to G$ such that $\varphi = \theta \circ \psi$ and $\theta^{-1}(G_1)$ and $\theta^{-1}(G_2)$ are not isomorphic.
\end{prop}

\begin{proof}
Our assumption implies that $\varphi^{-1}(G_1)$ and $\varphi^{-1}(G_2)$ have only finitely many open subgroups of any given index, so without loss of generality, them not being isomorphic is witnessed by a finite group $H_0$ which is a continuous quotient of $\varphi^{-1}(G_1)$ but not of $\varphi^{-1}(G_2)$ in view of \cite[Theorem 4.2.4]{Wilson}. We can thus find a finite group $H$ and a continuous surjective group homomorphism $\psi \colon \widetilde G \to H$ such that the quotient map from $\varphi^{-1}(G_1)$ to $H_0$ factors via the restriction of $\psi$ to $\varphi^{-1}(G_1)$, and $\varphi$ factors through $\psi$. It follows that $\psi(\varphi^{-1}(G_1))$ is not isomorphic to $\psi(\varphi^{-1}( G_2))$ because $H_0$ is a quotient of the first but not of the second.
\end{proof}

\begin{proof}[Proof of Proposition \ref{CorProfComp}]

We apply Proposition \ref{SmallGroupProp} with $\widetilde G$ the profinite completion of $\Gamma$, recalling that a finitely generated group has only finitely many subgroups of any given index, and that the closure of the image in $\widetilde G$ of a finite index subgroup $\Delta$ of $\Gamma$ is isomorphic to the profinite completion of $\Delta$.
\end{proof}

\section{Proving our Main Result}

We recall some aspects of Cornulier's construction as recast by Sambale in \cite{group-theoretic-cornulier}.

\begin{thm} \label{thm:explicit-cornulier}
Let $G$ be a finite group of order $n \ge 2$. Let $p > n$ be a prime number. Let $\widetilde P$ be the free group with $n$ generators, nilpotency class $n$, and exponent $p$. Denote the generators of $\widetilde P$ by $\{v_g\}_{g \in G}$, and note that $\widetilde P$ is finite. Let $Q = (\mathbb{F}_p^\times)^{|G|}$, with elements written as $a = (a_g)_{g \in G}$ for $a_g \in \mathbb{F}_p^\times$. Consider the action of $Q$ on $\widetilde P$ defined by $a(v_g) = v_g^{a_g}$, and the action of $G$ on $\widetilde P \rtimes Q$ defined by $h(v_g) = v_{hg}$ and $h\left((a_g)_{g \in G}\right) = (a_{h^{-1}g})_{g \in G}$. Then there is a subgroup $K \le [\widetilde P, \widetilde P]$ normalized by $\widetilde P$, $Q$ and $G$, such that for $P = \widetilde P / K$ and $N = P \rtimes Q$, the composition $G \to \operatorname{Aut}(N) \to \operatorname{Out}(N)$ is an isomorphism.
\end{thm}

\begin{proof}
For $|G| \ge 3$, this is the construction given in \cite[Theorem 8]{group-theoretic-cornulier}, which uses the quotient by $K = \langle[v_{hg_1}, v_{hg_2}, \dots, v_{hg_{n - 1}}, v_{hg_1}] \mid h \in G\rangle$, where $g_1, \ldots, g_n$ is some enumeration of the elements of $G$. For $|G| = 2$, we take $K = [\widetilde{P}, \widetilde{P}]$, and the result follows from \cite[Lemma 5]{group-theoretic-cornulier}. 
\end{proof}

Analyzing this construction, we can show that $N$ satisfies several special properties.

\begin{thm}\label{thm:cornulier-with-properties}
Let $G$ be a finite group. Then there exists a finite group $N$ such that
\begin{enumerate}
\item $G \cong \operatorname{Out}(N)$;
\item $Z(N) = 1$;
\item The only subgroup of $\operatorname{Aut}(N)$ isomorphic to $N$ is $\operatorname{Inn}(N)$;
\item There is a section $\operatorname{Out}(N) \to \operatorname{Aut}(N)$;
\item $N$ is supersolvable.
\end{enumerate}
\end{thm}

\begin{rmk}
We only need properties (1)--(3) for the proof of \cref{MainRes}.
\end{rmk}

\begin{proof}
We may assume $|G| \ge 2$. Let $N = P \rtimes Q$ be the group constructed in \cref{thm:explicit-cornulier} for $G$ and any prime number $p \ge |G| + 2$. Then property (1) holds by \cref{thm:explicit-cornulier}, property (5) holds as $N$ is an extension of a $p$--group by an abelian group of exponent $p - 1$, and property (4) holds since the isomorphism $G \to \operatorname{Out}(N)$ was defined by the composition $G \to \operatorname{Aut}(N) \to \operatorname{Out}(N)$. It remains to show properties (2) and (3).
    
We start with property (2). We have $P^\text{ab} \rtimes Q \cong (\mathbb{F}_p \rtimes \mathbb{F}_p^\times)^{|G|}$, a group with trivial center. Thus $Z(N) \subseteq [P, P]$. Let $z \in [P, P]$ be an element that is central in $N$. Assume, for the sake of contradiction, that $z$ is nontrivial. Let $k \le |G|$ be the unique integer such that $z$ is a nontrivial element of $\gamma_k(P) \setminus \gamma_{k+1}(P)$, where $\{\gamma_i(P)\}$ is the lower central series of $P$. Recall that $\gamma_k(P)/\gamma_{k+1}(P)$ is a vector space over $\mathbb{F}_p$, generated by the commutators of weight $k$ in the generators of $P$, and that the weight $k$ commutator operation defines a linear map of $\mathbb{F}_p$--vector spaces
\[ (P^\text{ab})^{\otimes k} \to \gamma_k(P)/\gamma_{k+1}(P). \]
Let $\zeta \in \mathbb{F}_p^\times$ be a multiplicative generator, and let $a \in Q$ be the element whose action on $P$ maps the generators of $P$ to their $\zeta$-th powers. Since $z$ is central in $N$, it follows that $a(z) = z$. However, by the multilinearity of the commutator operation, the action of $a$ on $\gamma_k(P)/\gamma_{k+1}(P)$ multiplies each weight $k$ commutator by $\zeta^k$, so $a(z) \equiv \zeta^k z \pmod {\gamma_{k+1}(P)}$. Since $1 \le k \le p - 2$, we have $\zeta^k \ne 1$, so $z$ must belong to $\gamma_{k+1}(P)$, a contradiction. Thus, we have shown $Z(N) = 1$.

We now show property (3). Let $\widetilde G = \operatorname{Aut}(N)$, so that we have a short exact sequence
\[1 \to N \to \widetilde G \to G \to 1. \]
We need to show that $\widetilde G$ has a unique subgroup isomorphic to $N$. Since $P$ is a normal $p$--Sylow subgroup of $\widetilde G$, it suffices to show that $\widetilde G/P \cong Q \rtimes G$ (which is the regular wreath product $C_{p-1} \wr G$) has a unique subgroup isomorphic to $Q$. Note that every $|G|$th power in $Q \rtimes G$ lies in the base group $Q$. If $H \le Q \rtimes G$ is a subgroup isomorphic to $Q$, then the set of $|G|$th powers of elements of $H$ forms a subgroup $H_0 \le Q$, and is isomorphic to the maximal subgroup of $C_{p-1}^{|G|}$ of exponent $\frac{p - 1}{\gcd(p - 1, |G|)}$. However, the only subgroup of $Q$ with this isomorphism type is the subgroup of $|G|$th powers in $Q$. Thus, $H$ contains all $|G|$th powers of elements in $Q$. Since $|G| < p - 1$, it follows that $H$ contains an element of $Q$ which is supported on a single coordinate. The centralizer of this element in $Q \rtimes G$ is equal to $Q$, but must contain $H$. This implies $H = Q$ and proves property (3).
\end{proof}

\begin{cor}\label{cor:extension-pullback-retraction}
Let $G$ be a finite group. Then there exists a finite group $\widetilde G$ and a surjective group homomorphism $\widetilde G \to G$ such that, denoting by $N$ its kernel, for any injection of $G$ into a finite group $S$, the following are equivalent:
\begin{enumerate}
\item The short exact sequence 
\[1 \to N \to \widetilde G \to G \to 1\]
is the pullback of some short exact sequence
\[1 \to N \to \widetilde S \to S \to 1\]
along the injection $G \to S$.
\item There is a homomorphism $r\colon S \to G$ that is a retraction of $G \to S$.
\end{enumerate}
Furthermore, the only subgroup of $\widetilde G$ that is isomorphic to $N$ is $N$ itself.
\end{cor}

\begin{rmk}
The implication (2) $\implies$ (1) holds for any surjective group homomorphism $\widetilde G \to G$ as we will see below in the proof.
\end{rmk}

\begin{proof}
We take $N$ as in \cref{thm:cornulier-with-properties} and let $\widetilde G = \operatorname{Aut}(N)$. Suppose that (1) holds, so we have a pullback diagram
\[ \xymatrix{ 1 \ar[r] & N \ar[r] \ar@{=}[d] & \widetilde G \ar[r] \ar[d] & G \ar[r] \ar[d] & 1 \\ 1 \ar[r] & N \ar[r] & \widetilde S \ar[r] & S \ar[r] & 1.}\] 
Each row of the diagram induces a homomorphism from its right-most group to $\operatorname{Out}(N)$,
so we have a commutative diagram
\[ \xymatrix{ G \ar[r] \ar[d] & \operatorname{Out}(N) \ar@{=}[d] \\ S \ar[r] & \operatorname{Out}(N). } \]
However, the top horizontal map is an isomorphism, so this diagram gives a retraction $S \to G$. Thus, we have shown that (1) implies (2). 

Now, suppose that (2) holds, so we have a retraction $r\colon S \to G$. Taking $\widetilde S$ to be the fiber product $\widetilde G \times_{G} S$ shows that (1) holds. The last part of the theorem is one of the properties of $N$ in \cref{thm:cornulier-with-properties}.
\end{proof}

\begin{proof}[Proof of \cref{MainRes}]
We take $\widetilde G$ as in Corollary \ref{cor:extension-pullback-retraction} with kernel $N$. We assume that $G_1, G_2 \le G$ are subgroups whose preimages $\widetilde G_1, \widetilde G_2 \le \widetilde G$ are isomorphic (by an isomorphism $\psi$). Let us show that this implies that $G_1, G_2$ are conjugate.
    
To that end, we build a finite group $S$ into which $G$ injects and apply Corollary \ref{cor:extension-pullback-retraction}. To achieve this, let $\Delta$ denote the (infinite) group obtained from $\widetilde G$ via an HNN extension using the isomorphism $\psi$. By construction, $\widetilde G$ injects into $\Delta$, and $\widetilde G_1$ and $\widetilde G_2$ are conjugate in $\Delta$. Now, $\Delta$ is a virtually free group and hence residually finite. Thus we obtain an injection $\widetilde G \hookrightarrow \widetilde S$, with $\widetilde S$ a finite quotient of $\Delta$.
    
Let $\widetilde \sigma \in \widetilde S$ be an element conjugating $\widetilde G_1$ to $\widetilde G_2$. Without loss of generality, we may assume that $\widetilde S = \langle\widetilde G, \widetilde \sigma\rangle$. Since $N^{\widetilde \sigma}$ is a subgroup of $\widetilde G$ isomorphic to $N$, we must have $N^{\widetilde \sigma} = N$. Thus $N \triangleleft \widetilde S$. Let $S = \widetilde S/N$ and denote the image of $\widetilde \sigma$ in $S$ by $\sigma$. Then $G$ injects into $S$, and $\sigma$ conjugates $G_1$ to $G_2$ in $S$. Since the extension $\widetilde G \to G$ is the pullback of $\widetilde S \to S$, it follows by Corollary \ref{cor:extension-pullback-retraction} that there is a retraction $r \colon S \to G$. Thus $r(\sigma)$ conjugates $G_1$ to $G_2$ in $G$ as required. 
\end{proof}

\section{Proof of \cref{thm:main:infinite}}

Let $S$ be a nonabelian simple group.
For $s\in S$ let $c_s \colon S \to S$ be the inner automorphism $c_s(x) = sxs^{-1}$ for $x \in S$.
The map $s \mapsto c_s$ is a group homomorphism from $S$ to $\mathrm{Aut}(S)$ which is injective since the center of $S$ is trivial, giving us an identification of $S$ with $\mathrm{Inn}(S)$.

\begin{lem} \label{TrivCentInn}
The subgroup $\mathrm{Inn}(S)$ of $\mathrm{Aut}(S)$ is normal and its centralizer is trivial.
\end{lem}

\begin{proof}
For every $a\in \mathrm{Aut}(S)$ and $s\in S$ we have $a c_s a^{-1} = c_{a(s)}$. Therefore if $a$ lies in the centralizer of $\mathrm{Inn}(S)$ we have $c_{a(s)}=c_s$ for every $s \in S$.  The aforementioned injectivity implies that $a(s) = s$ for every $s\in S$, so $a=\mathrm{id}_S$ as required.
\end{proof}

\begin{lem}\label{lem:centralizer-and-minimal-normal}
Let $P$ be a subgroup of $\mathrm{Aut}(S)$ that contains $\mathrm{Inn}(S)$. Then every nontrivial normal subgroup $N$ of $P$ contains $\mathrm{Inn}(S)$. Consequently $\mathrm{Inn}(S)$ is the unique minimal nontrivial normal subgroup of $P$.
\end{lem}

\begin{proof}
The intersection $N \cap \mathrm{Inn}(S)$ is a normal subgroup of the simple group $\mathrm{Inn}(S)$, so either $N \cap \mathrm{Inn}(S) =1$ or $\mathrm{Inn}(S) \leq N$. If the intersection of these normal subgroups of $P$ were trivial, we would get that $[N, \mathrm{Inn}(S)] \leq N \cap \mathrm{Inn}(S) = 1 $, namely $N$ is contained in the centralizer of $\mathrm{Inn}(S)$, contrary to \cref{TrivCentInn}. We conclude that $N$ contains $\mathrm{Inn}(S)$ as required.
\end{proof}

\begin{lem} \label{lem:rigidity}
Let $P,Q$ be subgroups of $\mathrm{Aut}(S)$ containing $\mathrm{Inn}(S)$, and let $\Psi \colon P\xrightarrow{}Q$ be a group isomorphism. Then there exists $a \in \mathrm{Aut}(S)$ such that $\Psi(p)=a p a^{-1}$ for every $p\in P$.
\end{lem}

\begin{proof}
It follows from \cref{lem:centralizer-and-minimal-normal} applied to $P$ and to $Q$ that $\Psi(I)=I$, so
\[ \Psi|_{\mathrm{Inn}(S)} \colon \mathrm{Inn}(S) \xrightarrow{} \mathrm{Inn}(S) \]
is an automorphism, hence it gives rise to an automorphism of $S$, which we denote by $a \in \mathrm{Aut}(S)$. That is, for every $s \in S$ we have $\Psi(c_s)=c_{a(s)}=a c_s a^{-1}$ so for every $i \in \mathrm{Inn}(S)$ we get that
\[ \Psi(i)=aia^{-1}. \]

We want to show that $\Psi(p) = apa^{-1}$, or equivalently $\Psi(p)^{-1} a p a^{-1} = \mathrm{id}_S$, for each $p \in P$. By \cref{TrivCentInn} it suffices to show that $\Psi(p)^{-1} a p a^{-1}$ commutes with every $i \in \mathrm{Inn}(S)$, or equivalently
\[ (a p a^{-1}) i (a p a^{-1})^{-1} = \Psi(p) i \Psi(p)^{-1}. \]
Indeed we have $a p a^{-1} i a p^{-1}a^{-1} = \Psi(p a^{-1} i a p^{-1}) = \Psi(p)\Psi(a^{-1} i a)\Psi(p)^{-1} = \Psi(p) i \Psi(p)^{-1}$.
\end{proof}

\begin{proof}[Proof of Theorem~\ref{thm:main:infinite}]
By Theorem~\ref{thm:DGG}, there exists a nonabelian simple group $S$ with $\mathrm{Out}(S) \cong G$. From now on we identify these two groups. Put $\widetilde G = \mathrm{Aut}(S)$, and let $\varphi \colon \widetilde G \to G$ be the quotient map of $\mathrm{Aut}(S)$ by $\mathrm{Inn}(S)$ onto $\mathrm{Out}(S)$. Toward a contradiction suppose that $\varphi^{-1}(G_1)$ and $\varphi^{-1}(G_2)$ are isomorphic. It follows from \cref{lem:rigidity} that they are conjugate in $\widetilde G$ because both contain $\mathrm{Inn}(S) = \ker \varphi$. We conclude that $G_1 = \varphi(\varphi^{-1}(G_1))$ and $G_2 = \varphi(\varphi^{-1}(G_2))$ are conjugate in $G = \varphi(\varphi^{-1}(G))$, contrary to our assumption.
\end{proof}

The strengthening of \cref{thm:main:infinite} where we demand, similarly to \cref{MainRes}, $\ker \varphi$ to be finite fails even if we allow $\varphi$ and $\widetilde G$ to depend on $G_1$ and on $G_2$. To see this, take $G$ to be the additive group $\mathbb Q$, and consider its nonconjugate subgroups $G_1=\mathbb{Z}$, $G_2=2\mathbb{Z}$. It's enough to check that for every infinite cyclic subgroup $C$ of $G$ we have $\varphi^{-1}(C) \cong \ker \varphi \times C$. Picking a preimage of a generator of $C$ under $\varphi$ we see that $\varphi^{-1}(C) \cong \ker \varphi \rtimes C$, with the group homomorphism $C \to \mathrm{Out}(\ker \varphi)$ being the restriction of the group homomorphism $G \to \mathrm{Out}(\ker \varphi)$. Since $G$ is divisible and $\mathrm{Out}(\ker \varphi)$ is finite (because $\ker \varphi$ is assumed finite), this homomorphism is trivial, so the semidirect product is isomorphic to a direct product. 

\section{Proof of Theorem \ref{C:1}}\label{Proof_C1}

The group $\Gamma$ is Kleinian, and it can be presented as:
$$\Gamma = \langle x,y,z~|~x^2, y^2, z^3, (yz)^2, (zx)^5, (xy)^3 \rangle.$$

\noindent  The Magma code shown below exhibits a surjective homomorphism from $\Gamma$ (denoted by {\tt{g}} in the code) onto $\mathrm{PSL}(2,29)$ (denoted by {\tt{h}}), lists the maximal subgroups of $\mathrm{PSL}(2,29)$ and then computes a presentation for the pre-images of the pair of $A_5$ subgroups that arise (as described above) as subgroups of index $203$ (denoted by {\tt{RR}} and {\tt{SS}}). Magma then shows that the pre-images have different numbers of index $5$ subgroups: in one case the subgroup has $1$ and the other subgroup has $5$. It follows that the pre-images have non-isomorphic profinite completions, so the theorem is a consequence of Proposition \ref{CorProfComp}.

\medskip

\begin{verbatim}
> g<x,y,z>:=Group<x,y,z|x^2,y^2,z^3,(y*z)^2,(z*x)^5,(x*y)^3>;
> print AbelianQuotientInvariants(g);
[]
> h:=PSL(2,29);
imgs := [h!(1, 9)(2, 19)(3, 16)(4, 26)(5, 11)(6, 20)(7, 8)(10, 13)(12, 21)
(14, 23)(15, 29)(22, 25)(24, 30)(27, 28),h!(1, 16)(2, 30)(3, 22)(4, 24)
(5, 18)(6, 13)(7, 23)(8, 28)(9, 25)(10, 29)(12, 17)(14, 27)
(15, 20)(19, 26),h!(1, 2, 26)(3, 20, 18)(4, 6, 14)(5, 15, 22)(7, 9, 29)(8, 21, 28)
(10, 25, 23)(11, 12, 17)(13, 24, 27)
(16, 19, 30)];
> e := hom< g->h | imgs >; e(g) eq h;
true
> M:=MaximalSubgroups(h);
> print M;
Conjugacy classes of subgroups
------------------------------

[1]     Order 28           Length 435
         Permutation group acting on a set of cardinality 30
         Order = 28 = 2^2 * 7
             (1, 23)(2, 13)(3, 7)(4, 29)(5, 25)(8, 11)(9, 26)(10, 21)(12, 15)
             (14, 27)(16, 20)(18, 28)(19, 24)(22, 30)
             (1, 26, 30, 20, 5, 27, 15)(2, 18, 6, 28, 13, 3, 7)
             (8, 19, 21, 10, 24, 11, 17)(9, 23, 12, 14, 25, 16, 22)
             (1, 9)(2, 24)(3, 21)(5, 25)(6, 17)(7, 10)(8, 28)
             (11, 18)(12, 30)(13, 19)(14, 20)(15, 22)(16, 27)(23, 26)
[2]     Order 30           Length 406
         Permutation group acting on a set of cardinality 30
         Order = 30 = 2 * 3 * 5
             (1, 3)(2, 26)(4, 13)(5, 9)(6, 29)(7, 12)(8, 28)
             (10, 19)(11, 16)(14, 18)(15, 25)(17, 24)(20, 22)(21, 27)
             (1, 16, 9, 19, 27)(2, 20, 12, 18, 13)(3, 21, 10, 5, 11)
             (4, 14, 7, 22, 26)(6, 30, 29, 25, 15)(8, 28, 24, 23, 17)
             (1, 24, 5)(2, 15, 4)(3, 9, 17)(6, 14, 20)(7, 12, 30)
             (8, 21, 19)(10, 27, 28)(11, 16, 23)(13, 25, 26)(18, 29, 22)
[3]     Order 406          Length 30
         Permutation group acting on a set of cardinality 30
         Order = 406 = 2 * 7 * 29
             (1, 4, 24, 7, 29, 20, 16)(2, 22, 25, 10, 6, 27, 19)
             (3, 17, 5, 21, 9, 23, 28)(8, 12, 14, 18, 15, 13, 11)
             (1, 19)(2, 4)(3, 11)(5, 12)(6, 20)(7, 25)(8, 17)
             (9, 18)(10, 29)(13, 28)(14, 21)(15, 23)(16, 27)(22, 24)
             (1, 22, 8, 16, 27, 17, 24, 19, 4, 23, 29, 18, 3, 28, 5, 
25, 6, 14, 30, 21, 20, 7, 12, 13, 11, 9, 10, 15, 2)
[4]     Order 60           Length 203
         Permutation group acting on a set of cardinality 30
         Order = 60 = 2^2 * 3 * 5
             (1, 7)(2, 6)(3, 5)(8, 30)(9, 29)(10, 28)(11, 27)
             (12, 26)(13, 25)(14, 24)(15, 23)(16, 22)(17, 21)(18, 20)
             (1, 17, 6)(2, 19, 10)(3, 11, 7)(4, 25, 12)(5, 9, 22)
             (8, 27, 13)(14, 20, 21)(15, 26, 16)(18, 29, 28)(23, 24, 30)
[5]     Order 60           Length 203
         Permutation group acting on a set of cardinality 30
         Order = 60 = 2^2 * 3 * 5
             (1, 29)(3, 26)(4, 24)(5, 11)(6, 30)(7, 17)(8, 15)
             (10, 12)(13, 22)(14, 18)(16, 21)(19, 28)(20, 25)(23, 27)
             (1, 17, 6)(2, 19, 10)(3, 11, 7)(4, 25, 12)(5, 9, 22)
             (8, 27, 13)(14, 20, 21)(15, 26, 16)(18, 29, 28)(23, 24, 30)
> a := M[4]`subgroup;
> A := a @@ e;
> R := ReduceGenerators (A);
> print Index(g,R);
203
> RR:=Rewrite(g,R);
> print AbelianQuotientInvariants (RR);
[]
> a := M[5]`subgroup;
> A := a @@ e;                         
> S := ReduceGenerators (A);       
>      print Index(g,S);
203
> SS:=Rewrite(g,S);                    
> print AbelianQuotientInvariants (SS);
[]
> l:=LowIndexSubgroups(RR,<5,5>);
> print #l;
1
> k:=LowIndexSubgroups(SS,<5,5>);
> print #k;
5
\end{verbatim}

\section{Acknowledgments} 

Mark Shusterman is The Dr. A. Edward Friedmann Career Development Chair in Mathematics.

Karshon and Shusterman's research is co-funded by the European Union (ERC, Function Fields, 101161909). Views and opinions expressed are however those of the authors only and do not necessarily reflect those of the European Union or the European Research Council. Neither the European Union nor the granting authority can be held responsible for them.

Lubotzky was supported by the European Research Council (ERC) under
the European Union's Horizon 2020 (N. 882751).

Lubotzky and Reid would like to thank the Isaac Newton Institute for Mathematical Sciences, Cambridge, for support and hospitality during the program Operators, Graphs and Groups,  where initial discussions on the contents of this paper started.  This work was supported by EPSRC grant no EP/R014604/1.

McReynolds would like to thank Ryan Spitler for conversations on Proposition \ref{SmallGroupProp}, and Reid thanks Eamonn O'Brien for Magma counseling, and Giles Gardam for conversations on the topic of this note.

We would like to thank an anonymous referee for several enlightening comments.


\end{document}